\documentclass[12pt,oneside,english]{amsart}
\usepackage[T1]{fontenc}
\usepackage[utf8]{inputenc}
\usepackage{geometry}
\usepackage{enumitem}
\usepackage{amstext}
\usepackage{amsthm}
\usepackage{amssymb}
\usepackage{graphicx}
\usepackage{todonotes}

\makeatletter
\numberwithin{equation}{section}
\numberwithin{figure}{section}
\theoremstyle{plain}
\newtheorem{thm}{\protect\theoremname}[section]
\theoremstyle{definition}
\newtheorem{problem}[thm]{\protect\problemname}
\theoremstyle{plain}
\newtheorem{cor}[thm]{\protect\corollaryname}
\theoremstyle{remark}
\newtheorem{rem}[thm]{\protect\remarkname}
\theoremstyle{definition}
\newtheorem{defn}[thm]{\protect\definitionname}
\theoremstyle{plain}
\newtheorem{prop}[thm]{\protect\propositionname}
\theoremstyle{plain}
\newtheorem{lem}[thm]{\protect\lemmaname}
\theoremstyle{remark}
\newtheorem*{claim*}{\protect\claimname}
\theoremstyle{definition}
\newtheorem{example}[thm]{\protect\examplename}

\usepackage{hyperref}

\makeatother

\usepackage{babel}
\providecommand{\claimname}{Claim}
\providecommand{\corollaryname}{Corollary}
\providecommand{\definitionname}{Definition}
\providecommand{\examplename}{Example}
\providecommand{\lemmaname}{Lemma}
\providecommand{\problemname}{Problem}
\providecommand{\propositionname}{Proposition}
\providecommand{\remarkname}{Remark}
\providecommand{\theoremname}{Theorem}

\begin{document}
\title{F-blowups and essential divisors for toric varieties}
\author{Enrique Chávez-Martínez, Daniel Duarte, and Takehiko Yasuda}
\address{Instituto de Ingeniería y Tecnología, UACJ, Ciudad Juárez, Chihuahua,
México}
\email{enrique.chavez@uacj.mx}
\address{Centro de Ciencias Matem\'aticas, UNAM, Campus Morelia, Morelia, Michoac\'an, M\'exico}
\email{adduarte@matmor.unam.mx}
\address{Department of Mathematics, Graduate School of Science, Osaka University
Toyonaka, Osaka 560-0043, JAPAN}
\email{yasuda.takehiko.sci@osaka-u.ac.jp}
\thanks{We thank Nobuo Hara, Yutaro Kaijima and an anonymous referee for helpful comments. D.~Duarte was partially supported by PAPIIT IN117523. T.~Yasuda
was supported by JSPS KAKENHI Grant Numbers JP18H01112, JP21H04994
and JP23H01070.}
\begin{abstract}
We investigate the relation between essential divisors and F-blowups,
in particular, address the problem whether all essential divisors
appear on the $e$-th F-blowup for large enough $e$. Focusing on
the case of normal affine toric varieties, we establish a simple sufficient
condition for a divisor over the given toric variety to appear on
the normalized limit F-blowup as a prime divisor. As a corollary,
we show that if a normal toric variety has a crepant resolution, then
the above problem has a positive answer, provided that we use the notion of essential divisors in the sense of Bouvier and Gonzalez-Sprinberg. We also provide an example
of toric threefold singularities for which a non-essential divisor
appears on an F-blowup. 
\end{abstract}

\maketitle
\global\long\def\bigmid{\mathrel{}\middle|\mathrel{}}%

\global\long\def\AA{\mathbb{A}}%

\global\long\def\CC{\mathbb{C}}%

\global\long\def\FF{\mathbb{F}}%

\global\long\def\GG{\mathbb{G}}%

\global\long\def\LL{\mathbb{L}}%

\global\long\def\MM{\mathbb{M}}%

\global\long\def\NN{\mathbb{N}}%

\global\long\def\PP{\mathbb{P}}%

\global\long\def\QQ{\mathbb{Q}}%

\global\long\def\RR{\mathbb{R}}%

\global\long\def\SS{\mathbb{S}}%

\global\long\def\ZZ{\mathbb{Z}}%

\global\long\def\bA{\mathbf{A}}%

\global\long\def\ba{\mathbf{a}}%

\global\long\def\bb{\mathbf{b}}%

\global\long\def\bd{\mathbf{d}}%

\global\long\def\bf{\mathbf{f}}%

\global\long\def\bg{\mathbf{g}}%

\global\long\def\bh{\mathbf{h}}%

\global\long\def\bj{\mathbf{j}}%

\global\long\def\bm{\mathbf{m}}%

\global\long\def\bp{\mathbf{p}}%

\global\long\def\bq{\mathbf{q}}%

\global\long\def\br{\mathbf{r}}%

\global\long\def\bs{\mathbf{s}}%

\global\long\def\bt{\mathbf{t}}%

\global\long\def\bv{\mathbf{v}}%

\global\long\def\bw{\mathbf{w}}%

\global\long\def\bx{\boldsymbol{x}}%

\global\long\def\by{\boldsymbol{y}}%

\global\long\def\bz{\mathbf{z}}%

\global\long\def\cA{\mathcal{A}}%

\global\long\def\cB{\mathcal{B}}%

\global\long\def\cC{\mathcal{C}}%

\global\long\def\cD{\mathcal{D}}%

\global\long\def\cE{\mathcal{E}}%

\global\long\def\cF{\mathcal{F}}%

\global\long\def\cG{\mathcal{G}}%

\global\long\def\cH{\mathcal{H}}%

\global\long\def\cI{\mathcal{I}}%

\global\long\def\cJ{\mathcal{J}}%

\global\long\def\cK{\mathcal{K}}%

\global\long\def\cL{\mathcal{L}}%

\global\long\def\cM{\mathcal{M}}%

\global\long\def\cN{\mathcal{N}}%

\global\long\def\cO{\mathcal{O}}%

\global\long\def\cP{\mathcal{P}}%

\global\long\def\cQ{\mathcal{Q}}%

\global\long\def\cR{\mathcal{R}}%

\global\long\def\cS{\mathcal{S}}%

\global\long\def\cT{\mathcal{T}}%

\global\long\def\cU{\mathcal{U}}%

\global\long\def\cV{\mathcal{V}}%

\global\long\def\cW{\mathcal{W}}%

\global\long\def\cX{\mathcal{X}}%

\global\long\def\cY{\mathcal{Y}}%

\global\long\def\cZ{\mathcal{Z}}%

\global\long\def\fa{\mathfrak{a}}%

\global\long\def\fb{\mathfrak{b}}%

\global\long\def\fc{\mathfrak{c}}%

\global\long\def\ff{\mathfrak{f}}%

\global\long\def\fj{\mathfrak{j}}%

\global\long\def\fm{\mathfrak{m}}%

\global\long\def\fp{\mathfrak{p}}%

\global\long\def\fs{\mathfrak{s}}%

\global\long\def\ft{\mathfrak{t}}%

\global\long\def\fx{\mathfrak{x}}%

\global\long\def\fv{\mathfrak{v}}%

\global\long\def\fD{\mathfrak{D}}%

\global\long\def\fJ{\mathfrak{J}}%

\global\long\def\fG{\mathfrak{G}}%

\global\long\def\fK{\mathfrak{K}}%

\global\long\def\fM{\mathfrak{M}}%

\global\long\def\fO{\mathfrak{O}}%

\global\long\def\fS{\mathfrak{S}}%

\global\long\def\fV{\mathfrak{V}}%

\global\long\def\fX{\mathfrak{X}}%

\global\long\def\fY{\mathfrak{Y}}%

\global\long\def\ru{\mathrm{u}}%

\global\long\def\rv{\mathbf{\mathrm{v}}}%

\global\long\def\rw{\mathrm{w}}%

\global\long\def\rx{\mathrm{x}}%

\global\long\def\ry{\mathrm{y}}%

\global\long\def\rz{\mathrm{z}}%

\global\long\def\a{\mathrm{a}}%

\global\long\def\AdGp{\mathrm{AdGp}}%

\global\long\def\Aff{\mathbf{Aff}}%

\global\long\def\Alg{\mathbf{Alg}}%

\global\long\def\age{\operatorname{age}}%

\global\long\def\Ann{\mathrm{Ann}}%

\global\long\def\Aut{\operatorname{Aut}}%

\global\long\def\B{\operatorname{\mathrm{B}}}%

\global\long\def\Bl{\mathrm{Bl}}%

\global\long\def\c{\mathrm{c}}%

\global\long\def\C{\operatorname{\mathrm{C}}}%

\global\long\def\calm{\mathrm{calm}}%

\global\long\def\center{\mathrm{center}}%

\global\long\def\characteristic{\operatorname{char}}%

\global\long\def\cjun{c\textrm{-jun}}%

\global\long\def\codim{\operatorname{codim}}%

\global\long\def\Coker{\mathrm{Coker}}%

\global\long\def\Conj{\operatorname{Conj}}%

\global\long\def\D{\mathrm{D}}%

\global\long\def\Df{\mathrm{Df}}%

\global\long\def\diag{\mathrm{diag}}%

\global\long\def\det{\operatorname{det}}%

\global\long\def\discrep#1{\mathrm{discrep}\left(#1\right)}%

\global\long\def\doubleslash{\sslash}%

\global\long\def\E{\operatorname{E}}%

\global\long\def\Emb{\operatorname{Emb}}%

\global\long\def\et{\textrm{et}}%

\global\long\def\etop{\mathrm{e}_{\mathrm{top}}}%

\global\long\def\el{\mathrm{e}_{l}}%

\global\long\def\Exc{\mathrm{Exc}}%

\global\long\def\FB{\mathrm{FB}}%

\global\long\def\FBtilde{\widetilde{\mathrm{FB}}}%

\global\long\def\FConj{F\textrm{-}\Conj}%

\global\long\def\Fitt{\operatorname{Fitt}}%

\global\long\def\Fr{\mathrm{Fr}}%

\global\long\def\Gal{\operatorname{Gal}}%

\global\long\def\GalGps{\mathrm{GalGps}}%

\global\long\def\GL{\mathrm{GL}}%

\global\long\def\Grass{\mathrm{Grass}}%

\global\long\def\H{\operatorname{\mathrm{H}}}%

\global\long\def\hattimes{\hat{\times}}%

\global\long\def\hatotimes{\hat{\otimes}}%

\global\long\def\hd{\mathbf{h}}%

\global\long\def\Hilb{\mathrm{Hilb}}%

\global\long\def\Hodge{\mathrm{Hodge}}%

\global\long\def\Hom{\operatorname{Hom}}%

\global\long\def\hyphen{\textrm{-}}%

\global\long\def\I{\operatorname{\mathrm{I}}}%

\global\long\def\id{\mathrm{id}}%

\global\long\def\Image{\operatorname{\mathrm{Im}}}%

\global\long\def\In{\mathrm{In}}%

\global\long\def\ind{\mathrm{ind}}%

\global\long\def\injlim{\varinjlim}%

\global\long\def\Inn{\mathrm{Inn}}%

\global\long\def\iper{\mathrm{iper}}%

\global\long\def\Iso{\operatorname{Iso}}%

\global\long\def\isoto{\xrightarrow{\sim}}%

\global\long\def\J{\operatorname{\mathrm{J}}}%

\global\long\def\Jac{\mathrm{Jac}}%

\global\long\def\kConj{k\textrm{-}\Conj}%

\global\long\def\KConj{K\textrm{-}\Conj}%

\global\long\def\Ker{\operatorname{Ker}}%

\global\long\def\Kzero{\operatorname{K_{0}}}%

\global\long\def\lcr{\mathrm{lcr}}%

\global\long\def\lcm{\operatorname{\mathrm{lcm}}}%

\global\long\def\length{\operatorname{\mathrm{length}}}%

\global\long\def\lwid{\operatorname{\mathrm{lwid}}}%

\global\long\def\M{\operatorname{\mathrm{M}}}%

\global\long\def\MC{\mathrm{MC}}%

\global\long\def\MHS{\mathbf{MHS}}%

\global\long\def\mld{\mathrm{mld}}%

\global\long\def\mod#1{\pmod{#1}}%

\global\long\def\Mov{\overline{\mathrm{Mov}}}%

\global\long\def\mRep{\mathbf{mRep}}%

\global\long\def\mult{\mathrm{mult}}%

\global\long\def\N{\operatorname{\mathrm{N}}}%

\global\long\def\Nef{\mathrm{Nef}}%

\global\long\def\nor{\mathrm{nor}}%

\global\long\def\nr{\mathrm{nr}}%

\global\long\def\NS{\mathrm{NS}}%

\global\long\def\op{\mathrm{op}}%

\global\long\def\orb{\mathrm{orb}}%

\global\long\def\ord{\operatorname{ord}}%

\global\long\def\P{\operatorname{P}}%

\global\long\def\PEff{\overline{\mathrm{Eff}}}%

\global\long\def\PGL{\mathrm{PGL}}%

\global\long\def\pt{\mathbf{pt}}%

\global\long\def\pur{\mathrm{pur}}%

\global\long\def\perf{\mathrm{perf}}%

\global\long\def\perm{\mathrm{perm}}%

\global\long\def\perpinner{(\perp)}%

\global\long\def\Pic{\mathrm{Pic}}%

\global\long\def\pr{\mathrm{pr}}%

\global\long\def\Proj{\operatorname{Proj}}%

\global\long\def\projlim{\varprojlim}%

\global\long\def\Qbar{\overline{\QQ}}%

\global\long\def\QConj{\mathbb{Q}\textrm{-}\Conj}%

\global\long\def\R{\operatorname{\mathrm{R}}}%

\global\long\def\Ram{\operatorname{\mathrm{Ram}}}%

\global\long\def\rank{\operatorname{\mathrm{rank}}}%

\global\long\def\rat{\mathrm{rat}}%

\global\long\def\Ref{\mathrm{Ref}}%

\global\long\def\rig{\mathrm{rig}}%

\global\long\def\red{\mathrm{red}}%

\global\long\def\reg{\mathrm{reg}}%

\global\long\def\rep{\mathrm{rep}}%

\global\long\def\Rep{\mathbf{Rep}}%

\global\long\def\sbrats{\llbracket s\rrbracket}%

\global\long\def\Sch{\mathbf{Sch}}%

\global\long\def\sep{\mathrm{sep}}%

\global\long\def\Set{\mathbf{Set}}%

\global\long\def\sing{\mathrm{sing}}%

\global\long\def\sm{\mathrm{sm}}%

\global\long\def\SL{\mathrm{SL}}%

\global\long\def\Sp{\operatorname{Sp}}%

\global\long\def\Span{\operatorname{Span}}%

\global\long\def\Spec{\operatorname{Spec}}%

\global\long\def\Spf{\operatorname{Spf}}%

\global\long\def\ss{\mathrm{ss}}%

\global\long\def\st{\mathrm{st}}%

\global\long\def\Stab{\operatorname{Stab}}%

\global\long\def\Supp{\operatorname{Supp}}%

\global\long\def\spars{\llparenthesis s\rrparenthesis}%

\global\long\def\Sym{\mathrm{Sym}}%

\global\long\def\T{\operatorname{T}}%

\global\long\def\tame{\mathrm{tame}}%

\global\long\def\tbrats{\llbracket t\rrbracket}%

\global\long\def\tl{\mathbf{t}}%

\global\long\def\top{\mathrm{top}}%

\global\long\def\tors{\mathrm{tors}}%

\global\long\def\tpars{\llparenthesis t\rrparenthesis}%

\global\long\def\Tr{\mathrm{Tr}}%

\global\long\def\ulAut{\operatorname{\underline{Aut}}}%

\global\long\def\ulHom{\operatorname{\underline{Hom}}}%

\global\long\def\ulInn{\operatorname{\underline{Inn}}}%

\global\long\def\ulIso{\operatorname{\underline{{Iso}}}}%

\global\long\def\ulSpec{\operatorname{\underline{{Spec}}}}%

\global\long\def\Utg{\operatorname{Utg}}%

\global\long\def\Unt{\operatorname{Unt}}%

\global\long\def\Var{\mathbf{Var}}%

\global\long\def\vector{\mathbf{v}}%

\global\long\def\Vol{\mathrm{Vol}}%

\global\long\def\Y{\operatorname{\mathrm{Y}}}%

\section{Introduction}

F-blowups, which were introduced in \cite{yasuda2012universal}, associate
a canonical sequence of blowups to a given variety in positive characteristic,
and are analogs of higher Nash blowups \cite{yasuda2007highernash}.
Let $X$ be a variety over a field $k$ of characteristic
$p>0$ and let $e$ be a positive integer. The $e$-th F-blowup of
$X$, denoted by $\FB_{e}(X)$, is defined to be the universal birational
flattening of the $e$-iterate $k$-linear Frobenius morphism $F^{e}\colon X_{e}\to X$.
It is also possible to describe it as a certain closed subscheme of
a Hilbert scheme, which parametrizes fat points or jets (see Remark
\ref{rem:FB in Hilb}). As is common in singularity theory, we often
prefer to work with normal varieties and hence consider the normalization
$\FBtilde_{e}(X)$ of $\FB_{e}(X)$, which we call a normalized F-blowup. 

The problem on F-blowups that first comes to one's mind might be whether,
given a variety $X$, $\FB_{e}(X)$ (or $\FBtilde_{e}(X)$) is smooth
for $e\gg0$. However, the answer is negative in general (see Remark
\ref{rem:G-Hilb}), although the answer is known to be affirmative
for curves and F-regular surfaces \cite{yasuda2012universal,hara2011splitting2,hara2012fblowups}
for instances. In light of this, we would like to raise another question.
To do so, we recall that an \emph{essential divisor} over $X$ means
a prime divisor ``appearing'' on every resolution of $X$. There are several variations on the notion of essential divisors, depending on what condition is put on resolutions and what ``appearing on a resolution'' means. See Appendix
for details on relations among these variations. 
 The following is the central question that this paper
seeks to address:
\begin{problem}
\label{prob:intro}For a normal variety $X$ and $e\gg0$, do all
essential divisors over $X$ appear on $\FB_{e}(X)$ as prime divisors?
\end{problem}

This problem may be regarded as an analog of the famous Nash problem
concerning relation between essential divisors and arc spaces (for
example, see \cite{plenat2015thenash}), since an F-blowup is a parameter
space of jets as mentioned above and jets are variations of arcs.
Prior research on F-blowups in the literature shows that, in some
instances, the problem has a positive answer; F-regular surface singularities
\cite{yasuda2012universal,hara2011splitting2,hara2012fblowups}, simple
elliptic singularities \cite{hara2015structure,hara2013fblowups},
tame Gorenstein quotient singularities of dimension three \cite{toda2009noncommutative,yasuda2012universal},
and more generally, Gorenstein linearly reductive quotient singularities
of dimension three \cite{liedtkenoncommutative}. Note that in these cases, where varieties in question are either two-dimensional or \(\QQ\)-factorial, all the notions of essential exceptional divisors coincide (Proposition \ref{prop:simplicial}).
Another source of
our motivation for the above problem is a result of Chávez-Martínez
\cite{chavez-martinez2021factorization} that a similar problem for
higher Nash blowups has a positive answer if $X$ is a normal surface
with an $A_{n}$-singularity in characteristic zero. In this sense,
Nash's two approaches to singularities, Nash blowups and arc spaces,
come together. 

In this paper, we study Problem \ref{prob:intro} in the case where
$X$ is a normal affine toric variety say defined by a cone $\sigma\subset N_{\RR}=\RR^{d}$
which is strongly convex, rational polyhedral and $d$-dimensional.
Here we follow the standard toric notation that $M$ and $N$ are
free abelian groups of rank $d$ which are dual to each other, and
$M_{\RR}$ and $N_{\RR}$ denote $M\otimes\RR$ and $N\otimes\RR$,
respectively. In fact, as far as toric varieties are concerned, we
can generalize F-blowups as follows. For each positive integer $l$,
there exists a Frobenius-like morphism $F^{(l)}\colon X_{(l)}\to X$
(regardless of the characteristic of the base field), which is induced
by the inclusion $M\hookrightarrow(1/l)M$ of lattices, where $M$
denotes the lattice of characters of a torus as usual in toric geometry.
Using Frobenius-like morphisms, we can similarly define the $(l)$-th
F-blowup of $X$, denoted by $\FB_{(l)}(X)$ in any characteristic.
If the characteristic is $p>0$, then $\FB_{(p^{e})}(X)=\FB_{e}(X)$.
The normalization $\FBtilde_{(l)}(X)$ of $\FB_{(l)}(X)$ is again
a normal toric variety and thus corresponds to a subdivision of the
cone $\sigma$. From \cite[Theorem 1.8]{yasuda2012universal}, the
sequence $(\FB_{(l)}(X))_{l\in\ZZ_{>0}}$ as well as the sequence
$(\FBtilde_{(l)}(X))_{l\in\ZZ_{>0}}$ stabilizes. The resulting blowups
obtained by stabilization are denoted by $\FB_{(\infty)}(X)$ and
$\FBtilde_{(\infty)}(X)$; we call them the \emph{limit F-blowup}
and \emph{normalized limit F-blowup} of $X$, respectively. The fan
corresponding to $\FBtilde_{(\infty)}(X)$ depends not on the base
field but only on the given cone $\sigma$. We denote this fan by
$\Delta$. The following theorem is our main result:
\begin{thm}[Theorem \ref{thm: inner lattice pt}]
Let $w\in\sigma\cap N$ be a primitive lattice vector in the cone
$\sigma$. Suppose that there exists $a\in(\sigma^{\vee})^{\circ}\cap M$
with $(a,w)=1$, where $\sigma^{\vee}\subset M_{\RR}$ is the dual
cone of $\sigma$ and $(\sigma^{\vee})^{\circ}$ is its interior.
Then, the ray $\RR_{\ge0}w$ belongs to the fan $\Delta$. Equivalently,
the divisor over $X$ corresponding to the ray $\RR_{\ge0}w$ appears
on the normalized limit F-blowup $\FBtilde_{(\infty)}(X)$ as a prime
divisor. 
\end{thm}

To apply this theorem to Problem \ref{prob:intro}, we introduce the
notion of \emph{moderate toric resolutions} (Definition \ref{def:moderate}).
Instead of giving its definition, we only mention here that it includes
minimal resolutions of normal toric surfaces and crepant toric resolutions. We also need to clarify which version of essential divisors we adopt; we adopt the version due to  Bouvier and Gonzalez-Sprinberg \cite{bouvier1995systeme}. 
We say that a toric divisor over \(X \) is \emph{BGS essential} if it appears on every toric resolution of \(X\) as a prime divisor.

\begin{cor}[Corollary \ref{cor:moderate-ess}]
If $X$ admits a moderate toric resolution, then every BGS essential 
divisor over \(X\) appears on $\FBtilde_{(\infty)}(X)$ as a prime divisor.
\end{cor}

\begin{rem}
Ishii and Kollár \cite{ishii2003thenash} considered notions of essential divisors, divisorially essential divisors, and toric divisorially essential divisors and showed that these three notions coincide if the given variety is a normal toric variety; they worked over an algebraically closed field, but this result holds over any base field (Proposition \ref{prop:ishii-kollar}). We see that every BGS essential divisor over a normal toric variety  \(X\) is either a toric prime divisor on \(X\) itself or an essential divisor in the sense of Ishii and Kollár. Moreover, the converse is true if \(X\) is \(\QQ\)-factorial (Proposition \ref{prop:simplicial}).
\end{rem}

We also address the following problem which is ``opposite'' to Problem
\ref{prob:intro}:
\begin{problem}
\label{prob:only-ess}Do \emph{only} essential divisors appear on
$\FBtilde_{(\infty)}(X)$?
\end{problem}

The answer to the problem was positive in all the previously known
examples except simple elliptic singularities. Therefore, one might
be inclined to expect this question to have an affirmative answer
for certain classes of mild singularities such as F-regular singularities
and normal toric singularities. However, we show that for some normal
toric threefold, a non-essential divisor appears on $\FBtilde_{(\infty)}(X)$
(Example \ref{exa:non-min wt}). 
\begin{rem}
The same problem as Problem \ref{prob:only-ess} for higher Nash blowups
has negative answer already for normal toric surfaces (see computation
in \cite{duarte2013nashmodification}).
\end{rem}

We end this introduction by noting that there is a
relation among our work and non-commutative resolutions (see Remark
\ref{rem:noncommutative}). 

Throughout the paper, we work over an arbitrary field \(k\) unless otherwise noted. 

The paper is organized as follows. In Section \ref{sec:Preliminaries},
we set up notation and recall some notions and known results that
are necessary for later discussion. In Section \ref{sec:Level-subpolyhedra},
we introduce the notion of level subpolyhedra, analysis of which plays
a key role in the proofs of our main results. In Section \ref{sec:quot-space},
we establish some results concerning to passing to a quotient space
of a vector space. Such an argument is used, when we treat a ray contained
in the boundary of the cone in question. In Section \ref{sec:Critical-arrows},
we introduce the notion of critical arrows and show a condition for
a given ray in $\sigma$ being contained in a cone of $\Delta$ of
certain dimension. In Section \ref{sec:Main-results}, we state and
prove our main results. In Appendix, we discuss relations among several notions of essential divisors.

\section{Preliminaries\label{sec:Preliminaries}}

\subsection{Lattices and vector spaces}

Let $M$ and $N$ be free abelian groups of rank $d$ which are dual
to each other and let $M_{\RR}:=M\otimes\RR$ and $N_{\RR}:=N\otimes\RR$.
The natural pairing of $a\in M_{\RR}$ and $w\in N_{\RR}$ is denoted
by $(a,w)$. We sometimes denote it by $w(a)$, regarding $w$ as
a linear function on $M_{\RR}$. 

Fixing a basis of $M$, we sometimes identify $M$ with $\ZZ^{d}$.
The dual basis then induces an identification $N=\ZZ^{d}$. Through
these identifications, the vector spaces $M_{\RR}$ and $N_{\RR}$
are given with the standard inner products. In turn, the inner products
enable us to talk about norms of vectors in these vector spaces as
well as orthogonality among vectors in the same vector space.

For a subset $W\subset N_{\RR}$, we define its \emph{(outer) orthogonal
subspace} $W^{\perp}\subset M_{\RR}$ by
\[
W^{\perp}:=\{a\in M_{\RR}\mid\forall w\in W,\,(a,w)=0\}.
\]
We define its \emph{inner orthogonal subspace} $W^{(\perp)}\subset N_{\RR}$
by
\[
W^{(\perp)}:=\{v\in N_{\RR}\mid\forall w\in W,\,\langle v,w\rangle_{N}=0\},
\]
where $\langle-,-\rangle_{N}$ denotes the inner product on $N_{\RR}$
given as above. When $W$ is a singleton $\{w\}$, we write them as
$w^{\perp}$ and $w^{(\perp)}$. Similarly, we define the outer and
inner orthogonal subspaces of a subset of $M_{\RR}$.

\subsection{BGS essential divisors}\label{subse:BGS}

Let $\sigma\subset N_{\RR}$ be a strictly convex, nondegenerate (that
is, $d$-dimensional), rational polyhedral cone and let $\sigma^{\vee}\subset M_{\RR}$
be its dual cone. Following \cite{yasuda2012universal}, we denote
$\sigma^{\vee}$ also by $A_{\RR}$ and define the monoid $A:=A_{\RR}\cap M$.
The affine toric variety $X=X_{\sigma}$ associated to $\sigma$ is
given by 
\[
X=X_{\sigma}=\Spec k[A].
\]
By a \emph{toric divisor over $X$}, we mean a prime divisor $E$ on \(Z\) for some toric birational morphism \(Z \to X\) of normal toric varieties.  We identify two toric divisors over \(X\), \(E\subset Z\) and \(E'\subset Z'\), if \(E\) and \(E'\) map to each other by the natural birational map between \(Z\) and \(Z'\).
We say that a toric divisor \(E\) over \(X\) is \emph{exceptional} if the image of \(E\) in \(X\) has codimension \(>1\). 

In this paper, we mainly consider the following version of essential divisors due to Bouvier and Gonzalez-Sprinberg \cite{bouvier1995systeme}:

\begin{defn}
We say that a toric divisor over \(X\) is \emph{BGS essential} if it appears on every toric resolution of \(X\) as a prime divisor.     
\end{defn}

A \emph{ray }in $N_{\RR}$ means a one-dimensional cone of the form
$\RR_{\ge0}w$ for a nonzero element $w\in N$. A \emph{ray of $\sigma$
}means a one-dimensional face of $\sigma$, while a \emph{ray in $\sigma$}
is a ray contained in $\sigma$. A \emph{primitive element} of $N$
means a nonzero element which is not divisible by any integer $n\ge2$.
Each ray contains a unique primitive element. As is well-known, there
are one-to-one correspondences
\begin{align*}
\{\text{rays contained in \ensuremath{\sigma}}\} & \leftrightarrow\{\text{primitive elements in \ensuremath{\sigma}}\}\\
 & \leftrightarrow\{\text{toric divisors over \ensuremath{X}}\}.
\end{align*}
By these correspondences, rays of $\sigma$ correspond to toric prime
divisors on $X$.

Let $B:=\sigma\cap N$, the ``monoid on the $N$-side,'' which should
not be confused with $A=\sigma^{\vee}\cap M$ mentioned above. We define a partial order on  $B$ as follows; for
$w,w'\in B$, $w\le w'$ if and only if $w'\in w+\sigma$. As
is well-known, the monoid $B$ has a unique minimal set of generators,
which is called the \emph{Hilbert basis }of $B$\emph{. }It is easy
to see that the Hilbert basis consists of the minimal elements of
$B\setminus\{0\}$ with respect to the partial order. 
\begin{prop}[{\cite[Theorem 1.10]{bouvier1995systeme}}]
\label{prop:bouvier-GS}Let $\rho\subset\sigma$ be a ray and let
$E_{\rho}$ be the corresponding (not necessarily exceptional) divisor
over $X$. Then, $E_{\rho}$ is BGS essential if and only if $\rho$ is spanned by an element of the Hilbert basis
of $B$.
\end{prop}

\subsection{F-blowups}

For each positive integer $l$, we put 
\begin{align*}
(1/l)\cdot M & :=\{m/l\in M_{\RR}\mid m\in M\}\text{ and}\\
(1/l)A & :=A_{\RR}\cap(1/l)\cdot M.
\end{align*}
We define 
\[
X_{(l)}:=\Spec k[(1/l)\cdot A],
\]
which is an affine toric variety isomorphic to $X$. The inclusion
$A\hookrightarrow(1/l)\cdot A$ of monoids define a toric morphism
\[
F^{(l)}\colon X_{(l)}\to X.
\]
If $k$ has characteristic $p>0$ and if $l=p^{e}$ for a positive
integer $e$, then $F^{(l)}$ is nothing but the $e$-iterate Frobenius
morphism of $X$. 
\begin{defn}
We define the\emph{ $(l)$-th F-blowup} of $X$, denoted by $\FB_{(l)}(X)$,
to be the universal birational flattening of $F^{(l)}$. Namely, it
is a variety $Y$ given with a proper birational morphism $f\colon Y\to X$
such that 
\begin{enumerate}
\item $(Y\times_{X}X_{(l)})_{\red}$ is flat over $Y$, and
\item if $f'\colon Y'\to X$ is another proper birational morphism satisfying
condition (1), then there exists a unique morphism $g\colon Y'\to Y$
with $f'=f\circ g$. 
\end{enumerate}
\end{defn}

In other words, $\FB_{(l)}(X)$ is the blowup at the coherent sheaf
$(F_{(l)})_{*}\cO_{X_{(l)}}$ (see \cite{oneto1991remarks,villamayoru.2006onflattening}
for more details on blowups at coherent sheaves).
\begin{rem}
\label{rem:FB in Hilb} For simplicity, suppose that \(k\) is algebraically closed. Then, there is another way to define F-blowups by
using Hilbert schemes. For each point $x\in X$, the scheme-theoretic
preimage $(F^{(l)})^{-1}(x)$ is a zero-dimensional closed subscheme
of $X_{(l)}$, and hence determines a point of the Hilbert scheme
$\Hilb(X)$ of zero-dimensional subschemes of $X$. The $(l)$-th
F-blowup $\FB_{(l)}(X)$ is defined to be the closure of $\{[(F^{(l)})^{-1}(x)]\mid x\in X_{\sm}(k)\}$
in $\Hilb(X_{(l)})$. If $k$ has characteristic $p>0$ and $l=p^{e}$,
then subschemes $Z\subset X_{(l)}$ corresponding to points of $\FB_{(l)}(X)=\FB_{e}(X)$
are fat points or jets in the sense that $Z_{\red}$ is a point. 
\end{rem}

\begin{rem}
\label{rem:G-Hilb}Suppose that \(k\) is algebraically closed. Suppose also that $G$ is a finite abelian group of
order coprime to the characteristic of $k$ and that it linearly acts
on an affine space $\AA_{k}^{d}$. The quotient variety $X=\AA_{k}^{d}/G$
is an affine toric variety associated to a simplicial cone $\sigma\subset N_{\RR}$.
There exists a natural proper birational morphism $\Hilb^{G}(\AA_{k}^{d})\to X$
from the $G$-Hilbert scheme \cite{ito1996mckaycorrespondence}. We
have an isomorphism $\Hilb^{G}(\AA_{k}^{d})\cong\FB_{(\infty)}(X)$
compatible with the birational morphisms to $X$ \cite{yasuda2012universal,toda2009noncommutative}.
There have been constructed an example of non-normal $G$-Hilbert
schemes \cite{craw2007moduliof} and an example of normal but singular
$G$-Hilbert schemes \cite{kkedzierski2011theghilbert}. These provide
an example of normal affine toric varieties $X$ such that $\FB_{(\infty)}(X)$
is non-normal and an example such that $\FB_{(\infty)}(X)$ is normal
(hence isomorphic to $\FBtilde_{(\infty)}(X)$) but singular. 
\end{rem}

If $T\subset X$ denotes the open torus, then $\FB_{(l)}(X)$ has
a natural $T$-action. There is a natural $T$-equivariant morphism
$\FB_{(l)}(X)\to X$ which is proper and birational. Thus, the normalization
$\FBtilde_{(l)}(X)$ of $\FB_{(l)}(X)$ is a normal toric variety
and obtained by a subdivision of the cone $\sigma$. From \cite[Th.\ 3.13]{yasuda2012universal},
the sequence of F-blowups,

\[
X=\FB_{(1)}(X),\FB_{(2)}(X),\FB_{(3)}(X),\dots.,
\]
stabilizes.
\begin{defn}
We define $\FB_{(\infty)}(X)$ to be $\FB_{(l)}(X)$ for sufficiently
large $l$ and call it the \emph{limit F-blowup of $X$}. We define
the \emph{normalized limit F-blowup }$\FBtilde_{(\infty)}(X)$ to
be its normalization. We denote by $\Delta$ the fan corresponding
to the toric proper birational morphism $\FBtilde_{(\infty)}(X)\to X$.
\end{defn}

The fan $\Delta$ has support $|\Delta|=\sigma$.
We have the following description of \(\Delta\) as a Gröbner fan. 

\begin{prop}[{$\Delta$ as a Gröbner fan \cite[Prop.\ 3.5]{yasuda2012universal}}]
A normalized F-blowup $\FBtilde_{(l)}(X)\to X$ corresponds to the
Gröbner fan of the ideal 
\[
\fa_{l}:=\langle x^{m}-1\mid m\in M\rangle_{k\left[(1/l)\cdot M\right]}\cap k\left[(1/l)\cdot A\right]\subset k\left[(1/l)\cdot A\right].
\]
In particular, the fan $\Delta$ is the Gröbner fan of $\fa_{l}$
for $l\gg0$.
\end{prop}

Because of this proposition, we often regard a point of the cone \(\sigma\) as a weight vector defining a partial order on the monoid \((1/l)A\), whose elements correspond to monomials in the ring  
\[
    k\left[(1/l)\cdot A\right] \subset k\left[(1/l)\cdot M\right]=k\left[x_1^{\pm 1/l}, \dots, x_d^{\pm 1/l}\right] ,
\]
as well as a partial order on \(A_\RR = \sigma \). For some arguments, we need to assume that a weight vector is in a general position, avoiding finitely many hyperplanes. We now make this condition more precise. Let us recall that we have fixed an identification $M=\ZZ^{d}$ so that norms
of vectors in $M_{\RR}$ are defined. 

\begin{defn}
Let $a_{1},\dots,a_{m}$ be the Hilbert basis of $A$. We define $\Theta:=\bigcup_{i}(A_{\RR}+a_{i})\subset A_{\RR}$
and define $D\in\RR_{>0}$ to be the diameter of $A_{\RR}\setminus\Theta$.
Namely, 
\[
D:=\sup\{|x-y|\mid x,y\in A_{\RR}\setminus\Theta\}.
\]
We define a finite set $M_{\le D}:=\{m\in M\mid|m|\le D\}$. 
\end{defn}

\begin{defn}
We define $\Delta^{*}$ to be the fan obtained by subdividing $\sigma$
by hyperplanes $m^{\perp}\subset\RR^{d}$, $m\in M_{\le D}$. We say
that a vector $w\in N_{\RR}$ is \emph{general }if for every $a\in M_{\le D}$,
$w(a)\ne0$. 
\end{defn}

\begin{lem}
$\Delta^{*}$ is a refinement of $\Delta$.
\end{lem}

\begin{proof}
In the proof of \cite[Th.\ 3.14]{yasuda2012universal}, we construct
a toric proper birational morphism $Y\to X$ such that for every $l$,
it factors as $Y\to\FB_{(l)}(X)\to X$. The construction there shows
that the proper birational morphism $Y\to X$ corresponds to $\Delta^{*}$.
This shows the lemma.
\end{proof}
Each element $w\in N_{\RR}$ induces a partial order $\ge_{w}$ on
$M_{\RR}$; for $a,b\in M_{\RR}$, $a\ge_{w}b$ if and only if $w(a)\ge w(b)$. 
\begin{lem}
\label{lem:Xi equiv}Let $w\in\sigma$ be a general element. Let $>$
be a total order on $\ZZ^{d}$ refining the partial order $\ge_{w}$.
Let $M_{\le D}^{w+}:=\{m\in M_{\le D}\mid w(m)>0\}$. The following
subsets of $A_{\RR}$ are identical:
\begin{enumerate}
\item $\Theta\cup\{a\in A_{\RR}\mid\exists b\in A_{\RR},\,a-b\in M_{\le D}^{w+}\}$
\item $\{a\in A_{\RR}\mid\exists b\in A_{\RR},a-b\in M\text{ and }w(a)>w(b)\}$
\item $\{a\in A_{\RR}\mid\exists b\in A_{\RR},a-b\in M\text{ and }a>b\}$
\end{enumerate}
\end{lem}

\begin{proof}
The first set is contained in the second one, and the second one is
contained in the third one. From \cite[Lem. \ 3.10]{yasuda2012universal},
the first one and the third one are identical. This shows the lemma. 
\end{proof}
\begin{defn}
\label{def:Xi}For a general $w\in\sigma$, we denote by $\Xi_{w}$
the subset of $A_{\RR}$ in Lemma \ref{lem:Xi equiv}. 
\end{defn}

\begin{prop}[{\cite[Lem.\ 3.11]{yasuda2012universal}}]
\label{prop:Initial-Xi}With the notation of Lemma \ref{lem:Xi equiv},
the initial ideal of $\fa_{l}$ with respect to $>$ is the monomial
ideal corresponding to $\Xi_{w}\cap(1/l)\cdot M$;
\[
\In_{>}(\fa_{l})=k\left[\Xi_{w}\cap(1/l)\cdot M\right].
\]
In particular, \(\In_{>}(\fa_{l})\) is independent of the total order \(>\) and depends only on \(l\) and \(w\). 
\end{prop}

\begin{defn}
\label{def:chamber}We define a \emph{chamber} of a fan to be the
interior of a nondegenerate cone belonging to the fan. 
\end{defn}

The following result will be used, in the proof of our main result,  to show that the fan \(\Delta\) is divided into sufficiently many cones around a certain point.

\begin{cor}
\label{cor:same-nondeg-Xi}Let $w_{1},w_{2}\in\sigma$ be general
elements. Then, $w_{1}$ and $w_{2}$ are in the same chamber of $\Delta$
if and only if $\Xi_{w_{1}}=\Xi_{w_{2}}$.
\end{cor}

\begin{proof}
Let us denote by \(\In_w(\fa_l)\) the initial ideal in Proposition \ref{prop:Initial-Xi}. 
Since \(\Delta\) is the Gröbner fan of \(\fa_l\), general vectors $w_{1}$ and $w_{2}$ are in the same chamber of $\Delta$
if and only if \(\In_{w_1}(\fa_l) = \In_{w_2}(\fa_l)\) if and only if 
\(
    \Xi_{w_{1}}\cap(1/l)M=\Xi_{w_{2}}\cap(1/l)M
\).
Now, the ``if'' part of the corollary is immediate. The ``only if'' part holds, since if  \(\Xi_{w_{1}}\ne \Xi_{w_{2}}\), then 
\(
    \Xi_{w_{1}}\cap(1/l)M\ne \Xi_{w_{2}}\cap(1/l)M
\) for \(l\gg 0\).
\end{proof}

\section{Level subpolyhedra of $A_{\protect\RR}$\label{sec:Level-subpolyhedra}}

To show our main result, we need to know when the subset \(\Xi_w \subset A_\RR\) changes as a general vector \(w \in \sigma\) varies. We accomplish this task by the geometry of ``arrows'' in the cone \(A_\RR\). In Sections \ref{sec:Level-subpolyhedra} to \ref{sec:Critical-arrows}, we develop tools for this purpose. 

We keep the notation of the last section. In particular, $\sigma\subset N_{\RR}$
is a strictly convex, nondegenerate and rational polyhedral cone.
We fix $w\in B\setminus\{0\}$ and let $\mu$ be the minimal
face of $\sigma$ containing $w$. Thus, $w$ is in the relative interior
$\mu^{\circ}$ of $\mu$. The dual face $\mu^{*}$ of $\mu$ is, by
definition, the face of $\sigma^{\vee}=A_{\RR}$ given by $\sigma^{\vee}\cap\mu^{\perp}$.
From \cite[Exercise 1.2.2]{cox2011toricvarieties}, we have $\mu^{*}=A_{\RR}\cap w^{\perp}$.
From $\dim\mu+\dim\mu^{*}=d$, we also deduce $\langle\mu^{*}\rangle_{\RR}=\mu^{\perp}$.
We introduce certain subsets of the cone $A_{\RR}=\sigma^{\vee}$,
which play a key role in later analysis. 
\begin{defn}
For $c\in\RR_{>0}$, we define the \emph{level $c$ subpolyhedron
}of $A_{\RR}$ to be $\Lambda_{(=c)}^{w}:=A_{\RR}\cap\{w=c\}$ and
the \emph{level $\le c$ subpolyhedron of $A_{\RR}$ }to be $\Lambda_{(\le c)}^{w}:=A_{\RR}\cap\{w\le c\}.$
Here we regard $w$ as a linear function on $M_{\RR}$ and by $\{w=c\}$,
we mean the set $\{a\in M_{\RR}\mid w(a)=c\}$. Similarly for $\{w\le c\}$.
We also define $\Lambda_{(<c)}^{w}:=A_{\RR}\cap\{w<c\}$, although
it is not a polyhedron. When the choice of $w$ is evident from the
context, we simply write $\Lambda_{(=c)}$, $\Lambda_{(\le c)}$ and
$\Lambda_{(<c)}$, omitting the superscript.
\end{defn}

We recall the notion of characteristic cones and its basic properties. 
\begin{defn}[{\cite[p.\ 100]{schrijver1986theoryof}}]
The \emph{characteristic cone} (also called the \emph{recession cone})
$C$ of a polyhedron $\Omega\subset M_{\RR}$ is defined by
\[
C:=\{y\in M_{\RR}\mid\forall x\in\Omega,\,x+y\in\Omega\}.
\]
\end{defn}

\begin{lem}[{\cite[(28), p.\ 106]{schrijver1986theoryof}}]
\label{lem:K+C}Let $C$ be the characteristic cone of a polyhedron
$\Omega$. Then, there exists a polytope $K$ such that $\Omega=K+C$. 
\end{lem}

\begin{lem}
\label{lem:polytope-chara}A polyhedron $\Omega$ is a polytope if
and only if its characteristic cone is the trivial cone $\{0\}$. 
\end{lem}

\begin{proof}
This easily follows from Lemma \ref{lem:K+C}.
\end{proof}
\begin{lem}
\label{lem:chara-cone}For every $c>0$, the characteristic cone of
$\Lambda_{(=c)}$ is the dual face $\mu^{*}$ of $\mu$.
\end{lem}

\begin{proof}
Let $C$ be the characteristic cone of $\Lambda_{(=c)}$. It is easy
to see $\mu^{*}\subset C$. Since $\Lambda_{(=c)}\subset A_{\RR}$,
$C$ is contained in the characteristic cone of $A_{\RR}$, which
follows from the description of the characteristic cone with inequalities
\cite[(4), p.\ 100]{schrijver1986theoryof}. Since the characteristic
cone of $A_{\RR}$ is itself, we have $C\subset A_{\RR}$. To show
$C\subset w^{\perp}$, we take arbitrary elements $y\in C$ and $x\in\Lambda_{(=c)}$.
From the definition of characteristic cone, we have $x+y\in\Lambda_{(=c)}$.
Thus, we have
\[
w(y)=w(x+y)-w(x)=c-c=0,
\]
which shows $C\subset w^{\perp}$. Since $w\in\mu^{\circ}$, from
\cite[Exercise 1.2.2]{cox2011toricvarieties}, we have
\[
C\subset A_{\RR}\cap w^{\perp}=\mu^{*},
\]
as desired. 
\end{proof}
\begin{cor}
\label{cor:polytope equiv conditions}Let $c>0$. The following are
equivalent: 
\begin{enumerate}
\item $\Lambda_{(=c)}$ is a polytope.
\item $w\in\sigma^{\circ}$.
\item $\mu=\sigma$. 
\item $\mu^{*}=\{0\}$. 
\end{enumerate}
\end{cor}

\begin{proof}
The equivalences (2) $\Leftrightarrow$ (3) and (3) $\Leftrightarrow$
(4) are obvious. The equivalence (1) $\Leftrightarrow$ (4) follows
from Lemmas \ref{lem:polytope-chara} and \ref{lem:chara-cone}. 
\end{proof}
\begin{cor}
\label{cor:Lambda =00003D K + mu^*}For each $c\ge0$, there exists
a polytope $K_{c}\subset M_{\RR}$ such that $\Lambda_{(=c)}=K_{c}+\mu^{*}$.
\end{cor}

\begin{proof}
This follows from Lemma \ref{lem:K+C}.
\end{proof}

\section{Passing to a quotient space\label{sec:quot-space}}

We keep fixing $w\in B\setminus\{0\}$ and denoting by
$\mu$ the minimal face of $\sigma$ containing $w$. In later sections,
we need to study the geometry of level polyhedra $\Lambda_{(=c)}$
and $\Lambda_{(\le c)}$ in relation to the lattice $M$. However,
it is more difficult to do so if $w$ is on the boundary of $\sigma$,
equivalently if $\mu\ne\sigma$, since level polyhedra are unbounded
in that case. To treat it, we adopt the strategy of passing to a quotient
space. 

We define $\overline{M}:=M/(\mu^{\perp}\cap M)$ and $\overline{M}_{\RR}:=M_{\RR}/\mu^{\perp}=\overline{M}\otimes\RR$.
We follow the terminology that overlined symbols are objects in the
quotient space $\overline{M}_{\RR}$. For example, $\overline{A_{\RR}},$
$\overline{\Lambda_{(=c)}}$ and $\overline{w^{\perp}}$ denote the
images of $A_{\RR}$, $\Lambda_{(=c)}$ and $w^{\perp}$ in $\overline{M_{\RR}}$,
respectively. The pairing $M_{\RR}\times N_{\RR}\to\RR$ induces a
non-degenerate pairing $\overline{M}_{\RR}\times\langle\mu\rangle_{\RR}\to\RR$.
Thus, we may regard $\langle\mu\rangle_{\RR}$ as the dual space of
$\overline{M}_{\RR}$. In particular, $w\in\mu$ is also considered
to be a nonzero linear function on $\overline{M}_{\RR}$. Under this
identification, the polyhedron $\overline{\Lambda_{(=c)}}$ is identical
to the intersection of $\overline{A_{\RR}}$ with the hyperplane $\{w=c\}\subset\overline{M}_{\RR}$.

\begin{lem}
\label{lem:bounded}$\overline{\Lambda_{(=c)}}$ is a polytope.
\end{lem}

\begin{proof}
From Corollary \ref{cor:Lambda =00003D K + mu^*}, we can write $\Lambda_{(=c)}=K+\mu^{*}$
with a polytope $K$. Since $\mu^{*}$ is in the kernel of $M_{\RR}\to\overline{M}_{\RR}$,
we have $\overline{\Lambda_{(=c)}}=\overline{K}$, which is bounded.
\end{proof}
\begin{lem}
\label{lem:lift vertex}For a vertex $\overline{P}\in\overline{\Lambda_{(=c)}}$,
there exists a vertex $P\in\Lambda_{(=c)}$ which maps to $\overline{P}$
by the quotient map $M_{\RR}\to\overline{M}_{\RR}$.
\end{lem}

\begin{proof}
Let $\overline{H}\subset\overline{M}_{\RR}$ be a supporting hyperplane
of $\overline{\Lambda_{(=c)}}$ with $\overline{H}\cap\overline{\Lambda_{(=c)}}=\overline{P}$
and let $H\subset M_{\RR}$ be its preimage. Then, $H$ is a supporting
hyperplane of $\Lambda_{(=c)}$. Let $I:=H\cap\Lambda_{(=c)}$, which
is a face of $\Lambda_{(=c)}$ and maps to $\overline{P}$ by the
quotient map $M_{\RR}\to\overline{M}_{\RR}$. Since $I$ is contained
in the strongly convex cone $A_{\RR}$, it contains no line. Since
a minimal face of a polyhedron is an affine subspace \cite[p.\ 104]{schrijver1986theoryof},
a minimal face of $I$ is a vertex. In particular, $I$ has a vertex,
which is automatically a vertex of $\Lambda_{(=c)}$ and maps to $\overline{P}$
by construction.
\end{proof}

\section{Critical arrows\label{sec:Critical-arrows}}

In this section, we give a certain geometric condition for the given
ray $\RR_{\ge0}w$ in $\sigma$ belonging to the fan $\Delta$. As
the key notion to state this condition, we introduce the notion of
\emph{critical arrows}:
\begin{defn}
By an \emph{arrow}, we mean an ordered pair $\alpha=(\bh_{\alpha},\bt_{\alpha})$
of distinct points $\bh_{\alpha},\bt_{\alpha}\in M_{\RR}$, which
we call the \emph{head }and \emph{tail, }respectively. The \emph{vector
associated to }an arrow $\alpha$ is defined to be $\bv_{\alpha}:=\hd_{\alpha}-\tl_{\alpha}$.
We say that an arrow $\alpha$ is \emph{integral }if $\vector_{\alpha}\in M$
and \emph{primitive }if $\vector_{\alpha}$ is a primitive element
of $M$. 
\end{defn}

\begin{defn}
Let $P,Q\in\RR^{d}$ and let $\alpha$ be an integral arrow. We say
that $P$ is $Q$-\emph{integral }if $P-Q\in\ZZ^{d}$. We say that
$P$ is $\alpha$\emph{-integral }if $P-\hd_{\alpha}$ is integral
or equivalently if $P-\tl_{\alpha}$ is integral. 
\end{defn}

The fixed identification $M=\ZZ^{d}$ induces an identification $N=\ZZ^{d}$
and hence the standard Euclidean metric on $N_{\RR}$. Recall that
a \emph{chamber} of a fan $\Sigma$ means the interior of a nondegenerate
cone of $\Sigma$ (Definition \ref{def:chamber}).
\begin{defn}
\label{def:tie-breaker}Let $w\in\sigma\setminus\{0\}$ and let $w^{\perpinner}\subset N_{\RR}$
be its inner orthogonal space. A \emph{tie-breaker }of $w$ is an
element $b\in w^{\perpinner}$ such that $w+b$
is in a chamber of $\Delta^{*}$ whose closure contains $w$. 
\end{defn}

\begin{defn}
\label{def:critical arrows}
Let $w\in\sigma\setminus\{0\}$ and let $b$ be a tie-breaker of \(w\). 
We say that an arrow $\alpha$ in $A_{\RR}$
is\emph{ $(w,b)$-critical} if for some $c\in\RR_{>0}$, 
\begin{enumerate}
\item $\alpha$ is integral,
\item $\hd_{\alpha},\tl_{\alpha}\in\Lambda_{(=c)}$,
\item there is no $\alpha$-integral point in $\Lambda_{(<c)}$,
\item $\tl_{\alpha}$ is the unique point where the restriction $b|_{\Lambda_{(=c)}}$ of the linear function $b\colon M_{\RR}\to\RR$
takes the minimum value.
\end{enumerate}
We say that an arrow is \emph{$w$-critical} if it is $(w,b)$-critical
for some tie-breaker $b$ of $w$.
\end{defn}

Note that from the second condition above, a $w$-critical arrow has
the associated vector $\vector_{\alpha}$ orthogonal to $w$.

\begin{lem}
\label{lem:tail not in Xi}If an arrow $\alpha$ in $A_{\RR}$ is
$(w,b)$-critical, then $\tl_{\alpha}\notin\Xi_{w+\delta b}$ for
every $\delta\in(0,1]$.
\end{lem}

\begin{proof}
Let  $\alpha$ be a $(w,b)$-critical arrow in  $A_{\RR}$.
From Corollary \ref{cor:same-nondeg-Xi},  the sets \(\Xi_{w+\delta b}\), \(\delta \in (0,1]\)
are identical to one another. Therefore, it suffices to show $\tl_{\alpha}\notin\Xi_{w+\delta b}$ for a sufficiently small \(   \delta >0 \). 
In turn, from  Definition \ref{def:Xi}, it suffices to show that there is no \(\alpha\)-integral point \( z \in A_\RR \) with \(\tl_{\alpha} \ge_{w+ \delta b} z\). From (3) of Definition \ref{def:critical arrows},  there is no such point in \(\Lambda_{(<c)}\). We now claim that for a sufficiently small  \(\delta>0\), every \(\alpha\)-integral point in \(\Lambda_{>c}:=A_\RR\cap \{w>c \}\) is larger than \(\bt_\alpha\) with respect to the partial order \(\ge_{w+\delta b}\). Indeed, for a sufficiently small  \(\epsilon>0\), there is no \(\alpha\)-integral point in \(\Lambda_{(<c+\epsilon)}\setminus\Lambda_{(\le c)}\). 
For \(0<\delta \ll \epsilon\), 
the hyperplane section 
\[
\{ w+\delta b = (w+\delta b)(\bt_\alpha) \}\cap A_\RR
\]
is sufficiently close to the 
the hyperplane section 
\[
\{ w = w(\bt_\alpha) (=c) \}\cap A_\RR
\]
and hence is contained in \(\Lambda_{(<c+\epsilon)}\). In this situation, all \(\alpha\)-integral points in \(\Lambda_{(>c)}\) are also in the halfspace 
\[
    \{ w+\delta b > (w+\delta b)(\bt_\alpha) \},
\]
which shows the claim.
Thus, the only remaining possibility is points on \(\Lambda_{(=c)}\). But,  (4) of Definition \ref{def:critical arrows} shows that \(\bt_\alpha\) is the unique minimal element in \(\Lambda_{(=c)}\) with respect to \(\ge_{w+\delta b}\), and hence there is no such point on \(\Lambda_{(=c)}\), either. We have completed the proof. 
\end{proof}

\begin{lem}
\label{lem:three-same-chamber}Let $w\in\sigma\setminus\{0\}$ and
let $b$ be a tie-breaker of $w$. Let $\tau$ be the cone of $\Delta$
with $w\in\tau^{\circ}$. Let $\langle\tau\rangle_{\RR}\subset N_{\RR}$
be the linear subspace generated by $\tau$ and let $V_{\tau}:=\langle\tau\rangle_{\RR}\cap w^{\perpinner}$.
For $u\in V_{\tau}$, if $0<\varepsilon\ll\delta\ll1$, then the three
elements 
\[
w+b,\,w+\varepsilon b+\delta u,\,w+\varepsilon b-\delta u
\]
are all in the same chamber of $\Delta$.
\end{lem}

\begin{proof}
We take a sufficiently small $\delta$ so that $w\pm\delta u\in\tau^{\circ}$.
If we add these two points with $\varepsilon b$ $(0<\varepsilon\ll\delta)$,
they move slightly and enter the chamber of $\Delta$ containing $w+b$.
See Figure \ref{fig:diagram-2}.
\end{proof}
\begin{figure}
\includegraphics[scale=0.5]{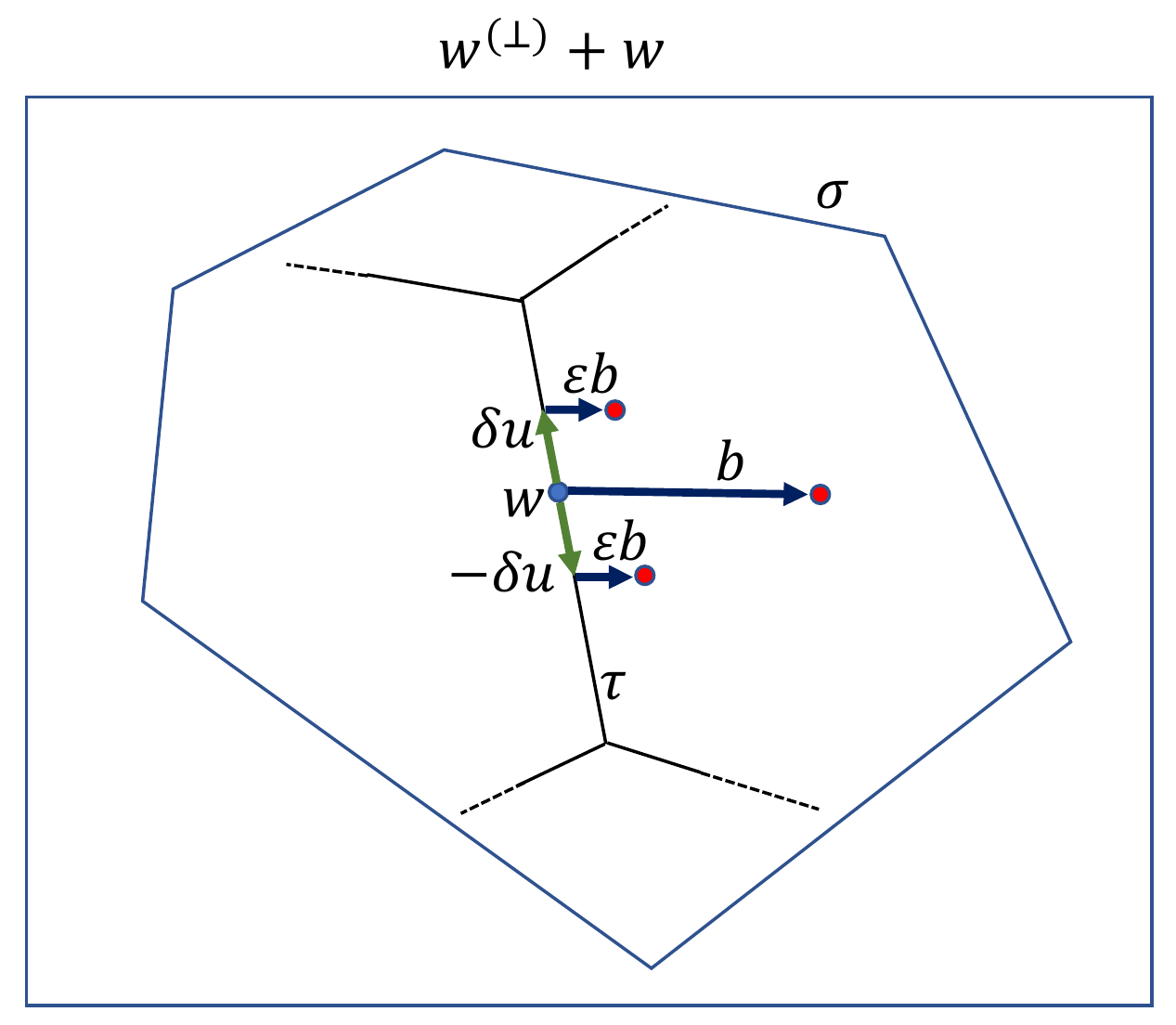}

\caption{Three points in the same chamber}
\label{fig:diagram-2}
\end{figure}

\begin{thm}
Let $0\ne w\in B$. Suppose that there exist $w$-critical
arrows $\alpha_{1},\dots,\alpha_{l}$ in $A_{\RR}$. Let $(\vector_{\alpha_{i}})^{\perp}\subset N_{\RR}$
be the (outer) orthogonal subspace of $\vector_{\alpha_{i}}$. Then,
the minimal cone of $\Delta$ containing $w$ is contained in $\bigcap_{i=1}^{l}(\vector_{\alpha_{i}})^{\perp}$.
\end{thm}

\begin{proof}
Let $\tau\in\Delta$ be the minimal cone containing $w$. For each
$i$, we take a tie-breaker $b_{i}$ of $w$ such that $\alpha_{i}$
is $(w,b_{i})$-critical. To show that $V_{\tau}\subset(\vector_{\alpha_{i}})^{\perp}$
by contradiction, we take an element $u\in V_{\tau}\setminus(\vector_{\alpha_{i}})^{\perp}$
and choose $\varepsilon$ and $\delta$ such that $0<\varepsilon\ll\delta\ll1$.
Then, $w\pm\delta u\in\tau$. From Lemma \ref{lem:tail not in Xi},
we have 
\[
\tl_{\alpha_{i}}\notin\Xi_{w+b_{i}}=\Xi_{w+\varepsilon b_{i}}.
\]
Since $|(\vector_{\alpha_{i}},\varepsilon b_{i})|\ll|(\vector_{\alpha_{i}},\delta u)|$,
real numbers $(\vector_{\alpha_{i}},\varepsilon b_{i}+\delta u)$
and $(\vector_{\alpha_{i}},\varepsilon b_{i}-\delta u)$ have different
signs, say $(\vector_{\alpha_{i}},\varepsilon b_{i}+\delta v)<0$.
Since $\bv_{\alpha_{i}}$ is orthogonal to $w$, we have 
\begin{align*}
 & (w+\varepsilon b_{i}+\delta u)(\bh_{\alpha_{i}})-(w+\varepsilon b_{i}+\delta u)(\bt_{\alpha_{i}})\\
 & =(\varepsilon b_{i}+\delta u)(\bh_{\alpha_{i}})-(\varepsilon b_{i}+\delta u)(\bt_{\alpha_{i}})\\
 & =(\varepsilon b_{i}+\delta u)(\bv_{\alpha_{i}})\\
 & <0.
\end{align*}
Thus, 
\[
\tl_{\alpha_{i}}\in\Xi_{w+\varepsilon b_{i}+\delta u},
\]
and $\Xi_{w+b_{i}}\ne\Xi_{w+\epsilon b_{i}+\delta u}$, which means
that $w+b_{i}$ and $w+\varepsilon b_{i}+\delta u$ are in different
chambers of $\Delta$. From Lemma \ref{lem:three-same-chamber}, $u\notin V_{\tau}$,
a contradiction as desired. We have showed that $V_{\tau}\subset(\vector_{\alpha_{i}})^{\perp}$.
It follows that $V_{\tau}\subset\bigcap_{i=1}^{l}(\vector_{\alpha_{i}})^{\perp}$.
We get
\[
\tau\subset\langle\tau\rangle_{\RR}=\RR w+V_{\tau}\subset\bigcap_{i=1}^{l}(\vector_{\alpha_{i}})^{\perp}.
\]
\end{proof}
\begin{cor}
\label{cor:lin-indep-ray}Let $0\ne w\in B$ and let $\mu$
be the minimal face of $\sigma$ containing $w$. Suppose that there
exist $w$-critical arrows $\alpha_{1},\dots,\alpha_{l}$ in $A_{\RR}$
such that $\overline{\vector_{\alpha_{1}}},\dots,\overline{\vector_{\alpha_{l}}}$
are linearly independent vectors of $\overline{M}_{\RR}=M_{\RR}/\mu^{\perp}$.
Then, the minimal cone of $\Delta$ containing $w$ has dimension
at most $\dim\mu-l.$ 
\end{cor}

\begin{proof}
Let $\tau\in\Delta$ be the minimal cone containing $w$. From the
above theorem, we have
\[
\langle\tau\rangle_{\RR}\subset\langle\mu\rangle_{\RR}\cap\bigcap_{i=1}^{l}(\vector_{\alpha_{i}})^{\perp}=(\mu^{\perp})^{\perp}\cap\bigcap_{i=1}^{l}(\vector_{\alpha_{i}})^{\perp}.
\]
Taking orthogonal subspaces, we get
\[
\tau^{\perp}\supset\mu^{\perp}+\sum_{i=1}^{l}\RR\vector_{\alpha_{i}}.
\]
From the assumption, $\dim\tau^{\perp}\ge\dim\mu^{\perp}+l$ and hence
\[
\dim\tau=\dim\langle\tau\rangle_{\RR}\le d-(\dim\mu^{\perp}+l)=\dim\mu-l.
\]
\end{proof}

\section{Main results\label{sec:Main-results}}

In this section, we prove our main result  and draw a few conclusions from it. 

\begin{lem}
\label{lem:exist min}We follow the notation in Section \ref{sec:quot-space}.
Let $\overline{P}\in\overline{\Lambda_{(=1)}}$ be a vertex and let
$\overline{L}\subsetneq\overline{M}_{\RR}$ be a proper subspace with
$(\overline{L}+\overline{P})\cap(\overline{\Lambda_{(=1)}})^{\circ}=\emptyset$.
Let $P$ be a vertex of $\Lambda_{(=1)}$ mapping to $\overline{P}$
and let $L\subset M_{\RR}$ be the preimage of $\overline{L}$. There
exists the minimum positive real number $c$ such that there exists
an integral arrow $\alpha$ in $\Lambda_{(=c)}$ with $\tl_{\alpha}=cP$
and $\vector_{\alpha}\notin L$. 
\end{lem}

\begin{proof}
We first consider the case where $w\in\sigma^{\circ}$. From Corollary
\ref{cor:polytope equiv conditions}, $\Lambda_{(=1)}$ is a polytope.
We have $\overline{M}_{\RR}=M_{\RR}$, $\overline{\Lambda_{(=1)}}=\Lambda_{(=1)}$,
$\overline{P}=P$ and $\overline{L}=L$. For $c_{0}\gg0$, there exists
an integral arrow $\alpha$ in $\Lambda_{(=c_{0})}$ with $\tl_{\alpha}=c_{0}P$
and $\vector_{\alpha}\notin L$. Let $e$ be the diameter of $\Lambda_{(=c_{0})}$.
To find the minimum $c$ as in the lemma, we only need to consider
integral arrows $\alpha$ with $|\bv_{\alpha}|\le e$. Let $v_{1},\dots,v_{n}\in(w^{\perp}\cap M)\setminus L$
be all the integral $w$-orthogonal arrows which are not contained
in $L$ and have norms at most $e$. Let $J\subset\RR_{\ge0}P$ be
the line segment connecting $0$ and $c_{0}P$. Now, the compact set
\[
J':=\bigcup_{i=1}^{n}\left(J\cap(A_{\RR}-v_{i})\right)
\]
is the set of points $cP\in J$, $c\le c_{0}$ such that there exists
an integral arrow $\alpha$ in $\Lambda_{(=c)}$ with $\tl_{\alpha}=cP$
and $\vector_{\alpha}\notin L$. Thus, the desired minimum value $c$
is the minimum value of the function $w|_{J'}\colon J'\to\RR$. Note
that this minimum value is positive, since for every $i$, we have
$v_{i}\notin A_{\RR}$ from construction. If the minimum is attained
on $J\cap(\Lambda_{(=c_{0})}-v_{i_{0}})$, then there exists an integral
arrow $\alpha$ in $\Lambda_{(=c)}$ with $\tl_{\alpha}=cP$ and $\vector_{\alpha}=v_{i_{0}}$. 

Next, we consider the general case. We begin with the following claim:
\begin{claim*}
We define integral arrows in $\overline{M}_{\RR}$ by using the lattice
$\overline{M}=M/(\mu^{\perp}\cap M)$. Then, the following are equivalent:
\end{claim*}
\begin{enumerate}
\item There exists an integral arrow $\alpha$ in $\Lambda_{(=c)}$ such
that $\tl_{\alpha}=cP$ and $\vector_{\alpha}\notin L$.
\item There exists an integral arrow $\overline{\alpha}$ in $\overline{\Lambda_{(=c)}}$
such that $\tl_{\overline{\alpha}}=c\overline{P}$ and $\vector_{\overline{\alpha}}\notin\overline{L}$. 
\end{enumerate}
\begin{proof}[Proof of Claim]
(1) $\Rightarrow$ (2): The image $\overline{\alpha}$ of $\alpha$
in $\overline{\Lambda_{(=c)}}$ satisfies $\tl_{\overline{\alpha}}=c\overline{P}$
and $\vector_{\overline{\alpha}}\notin\overline{L}$. 

(2) $\Rightarrow$ (1): Let $K\subset M_{\RR}$ be a polytope such
that $\Lambda_{(=c)}=K+\mu^{*}$. Let $H\in K$ be a lift of $\bh_{\overline{\alpha}}$.
There exists $a\in\mu^{\perp}$ such that $H+a$ is $cP$-integral.
Moreover, there exists $b\in\mu^{*}$ such that $a+b\in\mu^{*}$.
Then, 
\[
H':=H+a+b\in K+\mu^{*}=\Lambda_{(=c)}.
\]
The arrow $\alpha:=(H',cP)$ is an integral arrow in $\Lambda_{(=c)}$
which lifts $\overline{\alpha}$. From the construction, its associated
vector is not contained in $L$. 
\end{proof}
We go back to the proof of the lemma. Note that the polytope $\overline{\Lambda_{(=1)}}$
is obtained by cutting the image $\overline{A_{\RR}}\subset\overline{M}_{\RR}$
of $A_{\RR}$ by the hyperplane $\{w=1\}$. We see that $w$ is in
the interior of the dual cone $(\overline{A_{\RR}})^{\vee}\subset\langle\mu\rangle_{\RR}$.
The minimal face of $(\overline{A_{\RR}})^{\vee}$ containing $w$
is $(\overline{A_{\RR}})^{\vee}$ itself. From the special case of
the lemma which was discussed above, there exists the minimum $c>0$
satisfying condition (2). From the claim, there exists also the minimum
$c>0$ satisfying condition (1), which completes the proof. 
\end{proof}
\begin{lem}
\label{lem:no-pt-below}We keep the notation. Suppose that $(\Lambda_{(=1)})^{\circ}\cap M$
is not empty. Let $c$ be the minimum number as in Lemma \ref{lem:exist min}.
Then, there is no $cP$-integral point in $\Lambda_{(<c)}$. In particular,
an integral arrow $\beta$ in $\Lambda_{(=c)}$ with $\tl_{\beta}=cP$
is $w$-critical.
\end{lem}

\begin{proof}
The proof is by contradiction. We assume that there was an $cP$-integral
point $Q\in\Lambda_{(<c)}$ and derive a contradiction. Since $Q$
is $cP$-integral, $w(Q)=w(cP)-n$ for some positive integer $n$.
For $u\in(\Lambda_{(=1)})^{\circ}\cap M$, we have $R:=Q+nu\in(\Lambda_{(=c)})^{\circ}$.
Since $(L+P)\cap(\Lambda_{(=1)})^{\circ}=\emptyset$, $(L+cP)\cap(\Lambda_{(=c)})^{\circ}=\emptyset$.
This shows that $R-cP\notin L$. Since $R$ is in the interior of
$A_{\RR}$, for $0<\varepsilon\ll1$, 
\[
(R-\varepsilon cP,cP-\varepsilon cP)=(R-\varepsilon cP,(1-\varepsilon)cP)
\]
is an arrow in $A_{\RR}$. Its associated vector is 
\[
(R-\varepsilon cP)-(cP-\varepsilon cP)=R-cP\notin L.
\]
Since $(1-\varepsilon)c<c$, this contradicts the choice of $c$.

If $\beta$ is an integral arrow with $\tl_{\beta}=cP$ and if $b$
is a tie-breaker such that $b|_{\Lambda_{(=1)}}$takes the minimum
value only at $P$, then $\beta$ is $w$-critical.
\end{proof}
\begin{thm}
\label{thm: inner lattice pt}If $(\Lambda_{(=1)})^{\circ}\cap M\ne\emptyset$,
then the ray $\RR_{\ge0}w$ belongs to $\Delta$.
\end{thm}

\begin{proof}
Let $l:=\dim\mu$. From Corollary \ref{cor:lin-indep-ray}, it suffices
to show that there exist $w$-critical arrows $\alpha_{1},\dots,\alpha_{l-1}$
in $A_{\RR}$ such that $\overline{\vector_{\alpha_{1}}},\dots,\overline{\vector_{\alpha_{l-1}}}$
are linearly independent vectors in $\overline{M}_{\RR}$. To prove
this by contradiction, we suppose that there do not exist such arrows
$\alpha_{1},\dots,\alpha_{l-1}$. Then, there exists a proper subspace
$\overline{L}\subsetneq\overline{w^{\perp}}$ that $\overline{\vector_{\alpha}}\in\overline{L}$
for every $w$-critical arrow $\alpha$ in $A_{\RR}$. From Lemma
\ref{lem:bounded}, there exists a vertex $\overline{P}\in\overline{\Lambda_{(=1)}}$
such that $(\overline{L}+\overline{P})\cap(\overline{\Lambda_{(=1)}})^{\circ}=\emptyset$.
From Lemma \ref{lem:lift vertex}, there exists a vertex $P\in\Lambda_{(=1)}$
which is a lift of $\overline{P}$. Let $L\subset M_{\RR}$ be the
preimage of $\overline{L}$. From Lemma \ref{lem:no-pt-below}, there
exists a $w$-critical arrow $\alpha$ with $\vector_{\alpha}\notin L$.
This contradicts the construction of $\overline{L}$. 
\end{proof}
\begin{lem}
\label{lem:affine-intersect-rays}Suppose that there exists $w_{2},\dots,w_{d}\in B$
such that $w=w_{1},w_{2},\dots,w_{d}$ form a basis of $N$ and the
affine hyperplane spanned by $w_{1},\dots,w_{d}$ intersects with
every ray of $\sigma$. Then, $(\Lambda_{(=1)})^{\circ}\cap M\ne\emptyset$. 
\end{lem}

\begin{proof}
Let $a_{1},\dots,a_{d}\in M$ be the dual basis of $w_{1},\dots,w_{d}$.
Then, the affine hyperplane spanned by $w_{1},\dots,w_{d}$ is defined
by $\{\sum_{i}a_{i}=1\}$. The assumption shows that $\{\sum_{i}a_{i}=0\}$
is the supporting hyperplane of the origin as a vertex of $\sigma$.
This shows that $\sum_{i}a_{i}\in(A_{\RR})^{\circ}$. Since $(\sum_{i}a_{i},w)=1$,
we have $\sum_{i}a_{i}\in(\Lambda_{(=1)})^{\circ}$. 
\end{proof}
\begin{lem}
\label{lem:aff hyper}Let $\tau\subset N_{\RR}$ be a rational simplicial
nondegenerate cone and let $w_{1},\dots,w_{d}$ be the primitive vectors
of rays of $\tau$. Let $\psi_{\tau}\colon N_{\RR}\to\RR$ be the
unique linear function such that for every $i$, $\psi_{\tau}(w_{i})=1$.
Then, the following are equivalent:
\begin{enumerate}
\item The affine hyperplane spanned by $w_{1},\dots,w_{d}$ intersects with
every ray of $\sigma$.
\item The hyperplane $\psi_{\tau}=0$ is a supporting hyperplane of the
vertex of $\sigma$.
\end{enumerate}
\end{lem}

\begin{proof}
(1) $\Rightarrow$ (2): The affine hyperplane in (1) is defined by
$\psi_{\tau}=1$. Condition (1) shows that the restriction of $\psi_{\tau}$
to each ray of $\sigma$ take positive values except at the origin.
Condition (2) follows from the convexity of $\sigma$. 

(2) $\Rightarrow$ (1): Condition (2) shows that $\sigma\setminus\{0\}\subset\{\psi_{\tau}>0\}$.
This implies that each ray of $\sigma$ has a point where $\psi_{\tau}$
takes value 1. In turn, this means that Condition (1) holds.
\end{proof}
\begin{defn}
\label{def:moderate}Let $\Sigma$ be a subdivision of $\sigma$,
that is, a fan with $|\Sigma|=\sigma$. We say that $\Sigma$ is a
\emph{moderate smooth subdivision }of $\sigma$ if\emph{ }every nondegenerate
cone of $\Sigma$ is smooth and satisfies one of the equivalent conditions
in Lemma \ref{lem:aff hyper}. If this is the case, we call the corresponding
toric resolution $X_{\Sigma}\to X_{\sigma}$ a \emph{moderate toric
resolution.}
\end{defn}

\begin{prop}
Let $\Sigma$ be a moderate smooth subdivision of $\sigma$. Then,
every ray of $\Sigma$ is generated by a minimal element of $B\setminus\{0\}$,
that is, an element of the Hilbert basis of the monoid $B$.
In other words, toric prime divisors on a moderate toric resolution
of $X$ are precisely BGS essential divisors over $X$.
\end{prop}

\begin{proof}
The second assertion is a direct consequence of the first one and Proposition \ref{prop:bouvier-GS}. We will prove the first assertion.
On the contrary, suppose that there exists a primitive element $w\in B$
which is not minimal in $B\setminus\{0\}$, but generates a ray of
$\Sigma$. Let $\mu$ be the face of $\sigma$ containing $w$ in
its relative interior. The non-minimality of $w$ shows that $\dim\mu\ge2$.
There exists a distinct primitive element $v\in B\setminus\{0\}$
such that $w\in v+\sigma$. We have $v\in\mu$. We take the two-dimensional
subspace $U:=\langle v,w\rangle_{\RR}\subset N_{\RR}$. Let $\rho$
be the ray of the two-dimensional cone $\sigma\cap U$ such that $w\in\RR_{\ge0}v+\rho$.
The ray $\rho$ is contained in the boundary of $\mu$, while $\RR_{\ge0}w$
isn't. In particular, $\RR_{\ge0}w\ne\rho$. By an $\RR$-linear coordinate
change, we identify $U$ with an $xy$-plane $\RR^{2}$ so that the
cone $\RR_{\ge0}v+\rho$ becomes the first quadrant $(\RR_{\ge0})^{2}$.
From the construction, $w_{x}>0$ and $w_{y}\ge v_{y}$, where the
subscripts $x$ and $y$ mean the $x$- and $y$-coordinates of the
point in question. 

Let $\psi_{\Sigma}\colon\sigma=|\Sigma|\to\RR_{\ge0}$ be the piecewise
linear function such that for every nondegenerate cone $\tau\in\Sigma$,
$\psi_{\Sigma}|_{\tau}=\psi_{\tau}$. Let $F:=(\psi_{\Sigma}|_{\sigma\cap U})^{-1}(1)$,
which is the union of finitely many line segments. The two points
$v$ and $w$ are on $F$. Let $v=v_{0},v_{1},\dots,v_{n}=w$ be points
on $F$ which subdivide the path from $v$ to $w$ along $F$ into
the line segments $\overline{v_{i}v_{i+1}}$ for $i=0,1,\dots,n-1$.
Then, the sequence $v_{0,x},v_{1,x},\dots,v_{n,x}$ of $x$-coordinates
is strictly increasing. The moderateness of the subdivision $\Sigma$
implies that the line segments $\overline{v_{i}v_{i+1}}$ have negative
slopes. Thus, the sequence $v_{0,y},v_{1,y},\dots,v_{n,y}$ of $y$-coordinates
is strictly decreasing. In particular, $v_{y}=v_{0,y}>v_{n,y}=w_{y}$.
This contradicts the inequality $w_{y}\ge v_{y}$ mentioned above.
\end{proof}

\begin{example}
\label{exa:no moderate res}From \cite[Example 3.1]{bouvier1995systeme},
there exists a simplicial cone $\sigma$ which does not admit a smooth
subdivision by using only minimal elements of $B\setminus \{0\}$. The above
proposition shows that such a $\sigma$ does not admit any moderate
smooth subdivision.
\end{example}

\begin{cor}
\label{cor:moderate-ess}If $X$ admits a moderate toric resolution,
then every BGS essential divisor appears on $\FBtilde_{(\infty)}(X)$ as a prime divisor. 
\end{cor}

\begin{proof}
Let $\rho$ be any ray of $\Sigma$. There is a nondegenerate cone
$\tau\in\Sigma$ with $\rho\subset\tau$. From Theorem \ref{thm: inner lattice pt}
and Lemmas \ref{lem:affine-intersect-rays} and \ref{lem:aff hyper},
the corollary holds.
\end{proof}



\begin{prop}\label{prop:convex-moderate}
Suppose that there exists a smooth subdivision \(\Sigma\) of \(\sigma\) such that for every nondegenerate cone \(\tau \in \Sigma\), the simplex \(\tau \cap \{\psi_\tau =1\}\) is contained in the boundary of the convex hull \(C\) of \(B\setminus \{0\}\). Then, \(\Sigma\) is moderate.
\end{prop}

\begin{proof}
The convexity of \(C\) shows that $C\subset \{\psi_{\tau} \ge 1\}$. This shows that the affine hyperplane $\{\psi_{\tau}=1\}$ intersects every ray of $\sigma$. The proposition follows from the definition of moderate smooth subdivisions. 
\end{proof}

\begin{example}
When \(d=2\), Proposition \ref{prop:convex-moderate} and \cite[Th.\ 10.2.8]{cox2011toricvarieties} shows that the minimal resolution of \(X\) is a moderate toric resolution.
\end{example}

\begin{example}\label{exa:crepant-moderate}
Suppose that \(X\) is \(\QQ\)-Gorenstein, equivalently, that the primitive elements of rays of \(\sigma\) are on the same affine hyperplane say \(H\). Then,
a toric crepant resolution \(Y\to X\) is a moderate toric resolution. Indeed, for every toric prime divisor on \(Y\), the corresponding primitive element in \(B\) is on \(H\). Thus, the assumption of Proposition \ref{prop:convex-moderate} is satisfied.
\end{example}

\begin{cor}
\label{cor:crep-ess}If $X$ admits a toric crepant resolution, then
every BGS essential divisor appears on $\FBtilde_{(\infty)}(X)$ as a prime divisor.
\end{cor}

\begin{proof}
This follows from Example \ref{exa:crepant-moderate} and Corollary \ref{cor:moderate-ess}.
\end{proof}

\begin{example}
Every three-dimensional Gorenstein toric variety admits a crepant
resolution \cite[Prop.\ 11.4.19]{cox2011toricvarieties}. Every toric
variety (of any dimension) having only l.c.i.~singularities admits
a crepant resolution \cite{dais2001alltoric}.
\end{example}

\begin{rem}
Corollary \ref{cor:moderate-ess} does not mean that if $\Sigma$
is a moderate smooth subdivision of $\sigma$, then $\Delta$ is a
refinement of $\Sigma$. For example, for some finite abelian group
$G\subset\SL_{3}(\CC)$, the quotient $\AA_{\CC}^{3}/G$, which is
a toric variety, admits several toric crepant resolutions, which are
all moderate. The normalized limit F-blowup $\FBtilde_{(\infty)}(X)$
is also a toric crepant resolution. Thus, if $Y$ is a toric crepant
resolution of $\AA_{\CC}^{3}/G$ different from $\FBtilde_{(\infty)}(X)$,
then the morphism $\FBtilde_{(\infty)}(X)\to X$ does not factor through
$Y$.
\end{rem}

\begin{example}
\label{exa:non-min wt}Let $M=N=\ZZ^{3}$ and $M_{\RR}=N_{\RR}=\RR^{3}$
and suppose that the pairing $M_{\RR}\times N_{\RR}\to\RR$ is given
by the standard inner product of $\RR^{3}$. Consider the cone $\sigma\subset\RR^{3}$
spanned by $(1,0,0)$, $(0,1,0)$ and $(1,2,4)$, which appears in
\cite[p.\ 300]{atanasov2011resolving}. Its dual cone $\sigma^{\vee}=A_{\RR}$
is spanned by $(0,0,1)$, $(4,0,-1)$ and $(0,2,-1)$. Fix a weight
vector $w=(1,2,2)\in\sigma^{\circ}$. Since $(1,2,2)=(1,1,2)+(0,1,0)$,
$w$ is not minimal  in $B\setminus\{0\}$. The
triangle $\Lambda_{(=1)}$ has vertices
\[
P_{1}=(0,0,1/2),\,Q_{1}=(2,0,-1/2),\,R_{1}=(0,1,-1/2).
\]
Note that $(\Lambda_{(=1)})^{\circ}\cap\ZZ^{3}=\emptyset$. Indeed,
every element of $(\Lambda_{(=1)})^{\circ}$ is written as
\[
aP_{1}+bQ_{1}+cR_{1}\quad(a,b,c>0,\,a+b+c=1),
\]
which has the second coordinate in $(0,1)$. Thus, it cannot be an
element of $\ZZ^{3}$. The three vertices are integral to each other.
We claim that there is no point in $\Lambda_{(<1)}$ which is integral
relative to $P_{1}$, $Q_{1}$ and $R_{1}$. Indeed, if $U\in\Lambda_{(<1)}$
was integral to these points, then 
\[
w(U)\in[0,1)\cap\ZZ=\{0\}.
\]
The unique point of $\Lambda_{(<1)}$ with $w=0$ is the origin. A
direct computation shows that the origin is not integral to the three
points, which shows the claim. Thus, every edge of $\Lambda_{(=1)}$
given with either direction is a critical arrow. In particular, there
are two critical arrows in $\Lambda_{(=1)}$ whose associated vectors
are linearly independent. From Corollary \ref{cor:lin-indep-ray},
the ray $\RR_{\ge0}w$ belongs to $\Delta$, even though $w$ is not
minimal in $B\setminus\{0\}$. This example leads to the following problem.
\end{example}

\begin{problem}
Characterize 
divisors over $X$ which are not BGS essential, but appear on $\FBtilde_{(\infty)}(X)$. 
\end{problem}

\begin{rem}
Example \ref{exa:non-min wt} also shows that the existence of a moderate
toric resolution is a sufficient condition but not a necessary condition
for all BGS essential divisors appearing on $\FBtilde_{(\infty)}(X)$.
Thus, although some toric variety $X$ has no moderate toric resolution
(Example \ref{exa:no moderate res}), there is still a hope that all BGS
essential divisors may appear on $\FBtilde_{(\infty)}(X)$ for such
an $X$. It would be an interesting problem to look for a weaker sufficient
condition that can cover some toric varieties having no moderate toric
resolution. 
\end{rem}


\begin{rem}
\label{rem:noncommutative}Here we mention relation between our work
and noncommutative resolutions. As was observed in \cite{toda2009noncommutative},
if $X$ is an affine variety $\Spec R$ in characteristic $p>0$,
then the $e$-th F-blowup $\FB_{e}(X)$ may be viewed as the commutative
counterpart of the endomorphism ring $\mathrm{End}_{R}(R^{1/p^{e}})$.
When having finite global dimension, this ring is called a \emph{noncommutative
resolution} of $X=\Spec R$. This is the case for $e\gg0$, if $X$
has only tame quotient singularities \cite{toda2009noncommutative}
(see also \cite{liedtkenoncommutative} for a slight generalization
to quotients by linearly reductive finite group schemes) or if $X$
has only normal toric singularities \cite{vspenko2017noncommutative,faber2019noncommutative}.
The naive commutative counterpart of the last result would be the
statement that for a normal toric variety $X$ and for $e\gg0$, $\FB_{e}(X)$
is smooth. However, this is not true in general, as mentioned in Introduction
(see also Remark \ref{rem:G-Hilb}). It should be noted also that
the relation between commutative resolutions and non-commutative resolutions
tends to be less clear as the dimension increases. 
\end{rem}

\appendix

\section{Relations among several notions of essential divisors over a toric variety\label{appendix:Essential-divisors}}

Let $k$ be a field and let $Y$ be a normal
variety over $k$. For a proper birational morphism $f\colon Z\to Y$
from a normal variety, a prime divisor $E\subset Z$ is called a \emph{divisor
over }$Y$. We identify two divisors $E\subset Z$ and $E'\subset Z'$
over $Y$ if they correspond to each other by the natural birational
map between $Z$ and $Z'$. Equivalently, we identify them if they
induce the same valuation on the function field of $Y$. 

Let $E$ be a divisor over $Y$ and let $f\colon Z\to Y$ and $g\colon Z'\to Y$
be proper birational morphisms from normal varieties. Suppose that
$E$ is realized as a prime divisor of $Z$. The birational map $g^{-1}\circ f\colon Z\dasharrow Z'$
is defined on a nonempty open subset $E_{0}\subset E$. The \emph{center
}of $E$ on $Z'$ is defined to be the closure of $(g^{-1}\circ f)(E_{0})$
in $Z'$. This is independent of the birational model $Z$ on which
$E$ is realized as a prime divisor. We say that $E$ is \emph{exceptional
}if its center on $Y$ has codimension $>1$. More generally, if $Y'$
is a birational model of $Y$ and if $g\colon Z\dasharrow Y'$ is
the natural birational map, then the \emph{center }of $E$ on $Y'$
is defined to be $\overline{g(E)}$, the closure of $g(E)$. We say
that $E$ \emph{appears on $Y'$ }if the center of $E$ on $Y'$ has
codimension one. 

\begin{defn}[{\cite[Def.\ 2.3]{ishii2003thenash}}]
Let $Y$ be a normal variety and let $E$ be an exceptional divisor
over $Y$. We say that $E$ is an \emph{essential divisor }if for
every resolution $f\colon Z\to Y$, the center of $E$ on $Z$ is
an irreducible component of $f^{-1}(Y_{\sing})$ with $Y_{\sing}$
the singular locus of $Y$. We say that a resolution $f\colon Z\to Y$
is \emph{divisorial }if its exceptional set is of pure codimension
one. We say that $E$ is a \emph{divisorially essential divisor }if
$E$ appears on every divisorial resolution $f\colon Z\to Y$. 
\end{defn}

\begin{rem}\label{rem:Q-factorial}
Note that if \(Y\) is \(\QQ\)-factorial, then every resolution \(Z \to Y\) is divisorial \cite[Th.~1.5, Chap.~VI]{Kollar-rational-curves}
and hence there is no difference between ``essential'' and ``divisorially essential.''
\end{rem}

Let us now focus on the case of toric varieties. 
Let $\sigma\subset N_{\RR}$ be a strictly convex, nondegenerate (that
is, $d$-dimensional), rational polyhedral cone and let \(X\) be the associated affine normal toric variety. 

\begin{defn}
An exceptional divisor over $X$ is called a \emph{toric divisorially
essential divisor }over $X$\emph{ }if $E$ appears on every toric
divisorial resolution of $X$.
\end{defn}

\begin{prop}[{\cite[p.\ 609 and Cor.\ 3.17]{ishii2003thenash}}]
\label{prop:ishii-kollar}
For the toric variety $X$, the notions
of essential divisors, divisorially essential divisors, and toric
divisorially essential divisors coincide. 
\end{prop}

\begin{proof}
This was proved in the cited paper under the assumption that \(k\) is algebraically closed. We prove the proposition in the general case by reducing it to this special case. We have the following obvious implications:
\[
 \text{essential} \Rightarrow \text{divisorially essential} \Rightarrow    \text{toric divisorially essential}.
\]
Thus, we only need to show that every toric divisorially essential divisor is also essential. Let \(f\colon Y\to X\) be a resolution and let \(E\) be a toric divisorially essential divisor over \(X\). Note that \(E\) is geometrically irreducible, since it is toric. Let \(f_{\overline k}\colon Y_{\overline k} \to X_{\overline k}\) and \(E_{\overline k}\) denote their base changes to an algebraic closure  \(\overline{k}\)  of \(k\). Then, \(E_{\overline k}\) is a toric divisorially essential divisor over \(X\) and hence essential over \(X\). Thus, the center of \(E_{\overline k}\) in \(X_{\overline k}\) is an irreducible component of the exceptional set of \(f_{\overline k}\). Finally, this shows that the same thing is true over \(k\) and hence \(E\) is essential.
\end{proof}


\begin{defn}
We define 
\[
S_{\sigma}:=N\cap\left(\bigcup_{\tau:\text{singular face of \ensuremath{\sigma}}}\tau^{\circ}\right),
\]
where $\tau^{\circ}$ denotes the relative interior of $\tau$. 
\end{defn}

\begin{prop}[{\cite[p.\ 609 and Cor.\ 3.17 and 3.18]{ishii2003thenash}}]
\label{prop:ess-min}Let $\rho\subset\sigma$ be a ray and let $E_{\rho}$
be the corresponding divisor over $X$. Then, $\rho$ is spanned by
a minimal element of $S_{\sigma}$ if and only if $E_{\rho}$ is an
essential divisor. 
\end{prop}

\begin{proof}
Again,  the paper \cite{ishii2003thenash} treats the case where \(k\) is algebraically closed. However, the proof of Proposition \ref{prop:ishii-kollar} shows that the set of essential divisors stays the same by the base change to \(\overline k\). Thus, the assertion remains true over any field.
\end{proof}



If \(E_\rho \) is an exceptional divisor over \(X\) and if it appears on every (not necessarily divisorial) toric resolution of \(X\) as a prime divisor, then it is toric divisorially essential. Namely, a BGS essential divisor over \(X\) is either a toric divisorially essential divisor over \(X\) or a toric prime divisor on \(X\). 
Equivalently,
we have
\begin{equation*}\label{min-incl}
\begin{gathered}
    \{\text{minimal elements of \ensuremath{B\setminus\{0\}}}\}\subset\\
    \{\text{minimal elements of \ensuremath{S_{\sigma}}}\}\cup\{\text{primitive elements of the rays of \ensuremath{\sigma}}\}.
\end{gathered}
\end{equation*}
The opposite is not generally true, as the following example shows.

\begin{example}
Suppose that \(N=\ZZ^3\) and \(\sigma \subset \RR^3\) is generated by 
\[
(0,0,1), (1,0,1), (0,1,1), (1,1,1).
\] 
The toric variety \(X\) then have two small resolutions, corresponding to dividing the square with these four points as vertices along two diagonals respectively. This shows that there is no BGS essential exceptional divisor over \(X\) (all the BGS essential divisors are non-exceptional). 
On the other hand, \((1,1,2)\) is a minimal element of \(S_\sigma\) and hence the  divisor over \(X\) corresponding to the ray \(\RR_{\ge 0}(1,1,2)\) is essential. 
\end{example}

\begin{prop}\label{prop:simplicial}
Suppose that \(\sigma\) is simplicial. Then, every essential divisor over \(X\) is BGS essential. Namely, a toric divisor \(E\) over \(X\) is essential if and only if it is exceptional and BGS essential. 
\end{prop}

\begin{proof}
It suffices to show that every toric divisorially essential divisor is BGS essential. 
From the assumption, \(X\) is \(\QQ\)-factorial. This is well-known in characteristic zero (for example, see \cite[Lem.\ 14-1-1]{matsuki}). In arbitrary characteristic, \(X\) is the quotient of an affine space by an action of a finite diagonalizable group scheme. We can use this fact to show that \(X\) is \(\QQ\)-factorial. As  noted in Remark \ref{rem:Q-factorial}, every resolution of a \(\QQ\)-factorial variety is divisorial. This shows that every toric divisorially essential divisor is BGS essential, as desired.
\end{proof}

\bibliographystyle{alpha}
\bibliography{FB-EssDivs}

\end{document}